\documentclass[12pt]{amsart}

\usepackage{latexsym, amssymb, amsmath,ulem,soul}

\usepackage{amsfonts, graphicx}
\usepackage{graphicx,color}
\newcommand{\be}{\begin{equation}}
\newcommand{\ee}{\end{equation}}
\newcommand{\beq}{\begin{eqnarray}}
\newcommand{\eeq}{\end{eqnarray}}

\newtheorem{thm}{Theorem}[section]

\newtheorem{lma}{Lemma}[section]
\newtheorem{prop}{Proposition}[section]
\newtheorem{cor}{Corollary}[section]

\theoremstyle{remark}
\newtheorem{rem}{Remark}[section]
\numberwithin{equation}{section}

\def\be{\begin{equation}}
\def\ee{\end{equation}}
\def\bee{\begin{equation*}}
\def\eee{\end{equation*}}

\def\lf{\left}
\def\ri{\right}

\def\tR{ \mathcal{R}}

\def\supp{\mathrm{supp}}
\def\K{K\"ahler }

\def\KR{K\"ahler-Ricci }
\def\Ric{\text{\rm Ric}}
\def\Rm{\text{\rm Rm}}

\def\wt{\widetilde}
\def\wh{\widehat}
\def\la{\langle}
\def\ra{\rangle}
\def\p{\partial}

\def\e{\varepsilon}
\def\a{{\alpha}}
\def\b{{\beta}}

\def\R{\mathbb{R}}

\begin{document}

\title[]
{Some local maximum principles along Ricci flow}

 \author{Man-Chun Lee$^1$}
\address[Man-Chun Lee]{Department of Mathematics, Northwestern University, 2033 Sheridan Road, Evanston, IL 60208}
\email{mclee@math.northwestern.edu}
\thanks{$^1$Research partially supported by NSF grant DMS-1709894.}

\author{Luen-Fai Tam$^2$}
\address[Luen-Fai Tam]{The Institute of Mathematical Sciences and Department of
 Mathematics, The Chinese University of Hong Kong, Shatin, Hong Kong, China.}
 \email{lftam@math.cuhk.edu.hk}
\thanks{$^2$Research partially supported by Hong Kong RGC General Research Fund \#CUHK 14301517}

\renewcommand{\subjclassname}{
  \textup{2010} Mathematics Subject Classification}

\date{\today}

\begin{abstract}
In this work, we obtain a local maximum principle along the Ricci flow $g(t)$ under the condition that $\Ric(g(t))\le \a t^{-1}$ for $t>0$ for some constant $\a>0$. As an application, we will prove that     under this condition, various kinds of curvatures will still be nonnegative for $t>0$, provided they are nonnegative initially. These extend  the corresponding  known results for   Ricci flows on compact manifolds or on complete noncompact manifolds with {\it bounded curvature}.  By combining the above maximum principle with the Dirichlet heat kernel estimates, we also give a more direct proof of Hochard's \cite{Hochard2019} localized version of a maximum principle  by R. Bamler, E. Cabezas-Rivas, and B. Wilking \cite{BCRW} on the lower bound of different kinds of curvatures along the Ricci flows for $t>0$.

\end{abstract}


\maketitle


\section{Introduction}\label{s-intro}

 Given a Riemannian manifold $(M,g_0)$, the Ricci flow on $M$ is a  family of metrics $g(t)$ on $M$ satisfying
 \be\label{e-Ricci}
\left\{
  \begin{array}{ll}
   \partial_t g(t)&=-2\Ric(g(t)),   \hbox{\ \ on $M\times[0,T]$;} \\
  \    g(0)&= g_0.
  \end{array}
\right.
\ee
Here we denote $g(x,t)$ simply by $g(t)$. In this work, we always assume that the family is smooth in space and time.

Ricci flow is a useful tool in the study of structures of manifolds.  Ricci flow  is useful   because it tends to preserve    certain geometric structures. In many cases, the behaviour of a geometric structure is reflected by the behaviour of a scalar  function $\varphi$, which satisfies certain differential inequalities. One of the  simplest ways to obtain  useful information on $\varphi$, and hence  on the corresponding geometric structure, for $t>0$  is to apply maximum principles. In this work, we are interested in the following two frequently used differential inequalities along the Ricci flows:
\begin{equation}\label{heat-equ-RF-1}
\left(\partial_t-\Delta_{g(t)} \right) \varphi\leq L\varphi,
\end{equation}
for some continuous function $L(x,t)$ on $M\times [0,T]$ and
\begin{equation}\label{heat-equ-RF-2}
(\partial_t-\Delta_{g(t)}) \varphi\leq \mathcal{R}\varphi +K \varphi^2
\end{equation}
where $K$ is a positive constant and $\tR$ is the scalar curvature of $g(t)$. We will obtain two maximum principles for the two cases. The first one is the following:


\begin{thm}\label{t-MP1}
Let $(M^n,g(t)),t\in [0,T]$ be a smooth solution to the Ricci flow which is possibly incomplete.  Suppose
\begin{equation}\label{e-assumption-1}
\Ric(g(t))\le \a t^{-1}
\end{equation}
on $M\times (0,T]$ for some $\a>0$. Let $\varphi(x,t)$ be a continuous function on $M\times [0,T]$ which satisfies $\varphi(x,t)\leq \a t^{-1}$ on $M\times (0,T]$ and
\begin{equation}\label{evo-L}
\left(\frac{\partial}{\partial t}-\Delta_{g(t)} \right) \varphi \Big|_{(x_0,t_0)}\leq L(x_0,t_0)\varphi(x_0,t_0)
\end{equation}
{whenever $\varphi(x_0,t_0)>0$ in the sense of barrier}, for some continuous function $L(x,t)$ on $M\times [0,T]$ with $ L(x,t)\leq \a t^{-1}$.
Suppose $p\in M$ such that   $ B_{g(0)}(p,2)\Subset M$ and $\varphi(x,0)\leq 0$ on $B_{g(0)}(p,2)$. Then for any $l>\a+1$, there exists $\wh T(n,\a, l)>0$ such that for  $t\in[0,T\wedge \wh T]$,
$$\varphi(p,t)\le t^l.$$
\end{thm}
Here and below, we denote
$$
a\wedge b=:\min\{a,b\}.
$$
For the definition of {\it `in the sense of barrier'}, we refer readers to \cite[Chapter 18]{ChowBookIII}.

  Theorem \ref{t-MP1}  is known to be true if $M$ is compact without boundary or $M$ is noncompact and $g(t)$ is  a complete  solution with uniformly bounded curvature,   see \cite{Shi1997, HuangTam2018} for example. Nevertheless, there are
interesting results of the existence of the Ricci flows in which the initial metrics and the flows $g(t)$ may not be complete and may have  unbounded curvatures, see \cite{CabezasWilking2015,BCRW,ChauLee2019,GiesenTopping2011,
Lai2019,LeeTam2017,LeeTam2020,He2016,Hochard2016,
Simon2002,SimonTopping2017,Xu2013}.  However, most of the  
Ricci flow solutions  mentioned above satisfy  the condition \eqref{e-assumption-1}, which is invariant under parabolic rescaling.  This motivates us to obtain the maximum principles, Theorem \ref{t-MP1} and Theorem \ref{t-MP2} below.
   As an  immediate application, 
     Theorem \ref{t-MP1} will imply the preservation of nonnegativity of most known curvature conditions under the assumption that $ |\Rm(g(t))|\leq \a t^{-1}$ in the complete non-compact case. See Theorem \ref{t-preservation} for the full list of  the curvature conditions.  The theorem  also implies the preservation of the \K condition,  which is the first step in the use of the \KR flow to study  the uniformization   of complete noncompact \K manifolds with nonnegative bisectional curvature. See \cite{Shi1997,HuangTam2018} for more information.

In Theorem \ref{t-MP1},  the condition that $\varphi(0)\le 0$
is crucial and the analogous result  is not true if $\varphi$ is only  assumed to be bounded from above initially. This can be seen by considering Euclidean space with the time function $\varphi_\e(t)=(\frac{t+\e}{\e})^\a$.  The function satisfies $\varphi_\e(0)=1$ and \eqref{evo-L} with $L(x,t)=\a(t+\e)^{-1}$, but $\varphi_\e(t_0)\rightarrow +\infty$ as $\e\rightarrow 0$ for any fixed $t_0>0$.  Hence if the geometric quantity $\varphi(0)$ is only  assumed to be bounded from above,  one cannot  expect the analogous conclusion of Theorem \ref{t-MP1} holds.  However, if $\varphi$  satisfies \eqref{heat-equ-RF-2}, we have the following local upper estimates of $\varphi$ for a short time.  This   was first proved by Hochard in \cite[Proposition I.2.1]{Hochard2019}.

\begin{thm}\label{t-MP2} Let $g(t)$ be a smooth Ricci flow on $M^n\times[0,T]$ which is possibly incomplete.  Suppose $g(t)$ satisfies the following:
\bee
 \left\{
   \begin{array}{ll}
     |\Rm(g(x,t))|\leq \a t^{-1},& \hbox{for all $(x,t)\in M\times(0, T]$};\\
    \mathrm{inj}_{g(t)}(x) \geq \sqrt{\a^{-1}t},&\hbox{for all $(x,t)\in M\times(0, T]$ with}\\
    &\hbox {$B_{g(t)}(x,\sqrt{\a^{-1}t})\Subset M$;} \\
     \tR(g(0))\ge -\sigma T^{-1}   & \hbox{on $M$;} \\
      V_0(x,r)\le  r^n\exp(v_0 rT^{-\frac12})  & \hbox{ for all $ r>0$, and  $x\in M$ with $B_0(x,r)\Subset M$,}
   \end{array}
 \right.
\eee
for some $\a,v_0>0,\sigma\ge0$ where $\tR(g(0))$ is the scalar curvature of $g(0)$.
   Let $\varphi$ be a nonnegative continuous function on $M\times [0,T]$ satisfying \eqref{heat-equ-RF-2} in the sense of barrier.  Assume that
\bee
\left\{
  \begin{array}{ll}
    \varphi(0)\leq \delta, & \hbox{on $M$ for some $\delta>0$;} \\
    \varphi(t)\leq \a t^{-1}, & \hbox{on $M\times(0,T]$.}
  \end{array}
\right.
\eee
 Suppose $p\in M$ is a point  such that
 $B_{g_0}(p, 3RT^{\frac12})\Subset M$ for some $R>0$.   Then
$$
\varphi(p,t)\le C((RT^\frac12)^{-2}+\delta)
$$
for $t\in [0, T]$ for some constant $C>0$ depending only on  $n, K, \a,v_0,\sigma$.
\end{thm}

\begin{rem} By volume comparison, if $\Ric(g(0))\ge -\frac1n \sigma T^{-1}$, then the conditions on $\tR(g(0))$ and $V_0(x,r)$ in the theorem will be satisfied for $\sigma$ and for some $v_0>0$. 
\end{rem}

 Maximum principle  for the evolution equation \eqref{heat-equ-RF-2}  along the Ricci flow was first considered by R. Bamler, E. Cabezas-Rivas and B. Wilking in \cite{BCRW}. In particular, they  showed that if $\varphi$ is the negative part of the smallest eigenvalue of $\Rm(x,t)$ with respect to  certain curvature cones, then $\varphi$ satisfies \eqref{heat-equ-RF-2} in the barrier sense.  They proved that for the  Ricci flow  $g(t)$ on a compact manifold  or on a complete noncompact manifold  with bounded curvature, if $g(t)$ and $\varphi(t)$ satisfy the conditions in Theorem \ref{t-MP2} and $\varphi(0)\le \delta$, then $\varphi(t)\le  C\delta$ within a short time-interval $[0,T_0]$ for some constant $C>0$ both depending only on $n, \a, \sigma$ and $v_0$.  Theorem \ref{t-MP2}  is a localized version of this result. In \cite{Hochard2019}, Hochard  proved
   Theorem \ref{t-MP2}
 by obtaining estimates of the heat kernels together with their gradients for the backward heat equation on a nested sequences of domains. In this work, we will give a more direct proof by  combining the Dirichlet heat kernel estimates on a fixed $g(0)$-geodesic ball   with Theorem \ref{t-MP1}.  The proof is in the   spirit of work \cite{BCRW}.

 The localized   maximum principle   Theorem \ref{t-MP2}   is particularly useful when we consider the partial Ricci flow.  Combining the maximum principle with the partial Ricci flow method \cite{Hochard2016,SimonTopping2017}, Lai \cite{Lai2019}  constructed  a complete Ricci flow solution starting from a complete non-collapsed metric which is of almost weakly $\mathbf{PIC}_1$, and remains almost weakly $\mathbf{PIC}_1$ for a short time.
 In \cite{McLeodTopping2019},   McLeod and   Topping  combined Theorem \ref{t-MP2}, Lai's Ricci flow solutions \cite{Lai2019} and the  techniques developed in their earlier work  \cite{McLeodTopping2018} to obtain a smooth structure on the noncollapsed $\mathbf{IC}_1$-limit space. In \cite{LeeTam2020}, the authors used Theorem \ref{t-MP2}  to construct a local \KR flow starting from a non-collapsed \K manifold with almost nonnegative curvature and improve a result of Liu  \cite{Liu2018} on the complex structure of the corresponding Gromov-Hausdorff limit of this class of \K manifolds. See the recent work by Lott \cite{Lott2020} for further development.

The paper is organized as follows: In section 2, we will collect some useful lemmas  which allow us to compare $g(0)$-geodesic balls and $g(t)$-geodesic balls. In section 3, we will give a proof of Theorem \ref{t-MP1} and a unified proof for preservation of nonnegativity of some curvature conditions. In section 4, we will obtain Dirichlet heat kernel estimates for the backward heat equation and give a proof of Theorem \ref{t-MP2}.

{\it Acknowledgement}:  The authors would also like to thank the referee for the useful comments.

\section{Shrinking and expanding balls Lemmas}\label{s-balls}

Let $(M^n,g(t))$ be a Ricci flow defined on $M\times[0,T]$.
 Since $g(t)$ may not be complete, we use the following convention:
Let $(M,g)$ be a Riemannian manifold without boundary which may be incomplete. Let $x\in M$, $r>0$.  If $\exp_{g,x}$ is defined on the ball  $B(r)$   of radius $r$  in the tangent space $T_x(M)$ with center at the origin, then we denote $B_g(x,r)=\text{Image}(\exp_{g,x}(B(r))$. We say that $B_g(x,r)\Subset M$ if it is compactly contained in $M$. We say that the injectivity radius $\mathrm{inj}_g(x)$ of $x$ satisfies $\mathrm{inj}_g(x)\ge \iota_0$, if $B_g(x,\iota_0)\Subset M$ and $\exp_{g,x}$  is a diffeomorphism from the ball of radius $\iota_0$ onto its image $B_g(x,\iota_0)$.
Observe that if $B_g(x,r)\Subset M$, then any point in $B_g(x,r)$ can be joined to $x$ by a minimizing geodesic in $M$. If $B_g(x,2r)\Subset M$, then any two points in $B_g(x,r)$ can be joined by a minimizing geodesic lying inside $B_g(x,2r)$. In this case, the distance function is well-defined on $B_g(x,r)$. We will omit the subscript $g$ when the content is clear. In the rest of the work, we denote the ball of radius $r$ with respect to $g(t)$ by $B_t(x,r)$ and its volume $\mathrm{Vol}_{g(t)}(B_t(x,r))$ by $V_t(x,r)$. Moreover, the distance function with respect to $g(t)$ is denoted by $d_t$.

Since $g(t)$ is not necessarily complete, it is important to compare balls with respect to $g(t)$  at different time. Some   basic results on this will be used later.
The first one is the     following
 shrinking balls Lemma by Simon-Topping   \cite[Corollary 3.3]{SimonTopping2016}:

\begin{lma}\label{l-balls}
 There exists a constant $\beta=\beta(n)\geq 1$ depending only on $n$ such that the following is true. Suppose $(M^n,g(t))$ is a Ricci flow for $t\in [0,T]$ and $x_0\in M$   with $ B_0(x_0,r)\Subset M$ for some $r>0$.  Suppose $g(t)$ satisfies $\Ric(g(t))\leq (n-1)a/t$ on $B_0(x_0,r)$  for some $a>0$ for all $t\in (0,T]$. Then
$$B_t\left(x_0,r-\beta\sqrt{a t}\right)\subset B_0(x_0,r), $$
and in general for $0<s<t<T$,
$$
B_t\left(x_0,r-\beta\sqrt{a t}\right)\subset B_s\left(x_0,r-\beta\sqrt{a s}\right).
$$
In particular,
$$
d_t(y,x_0)\ge d_{s}(y,x_0)-\beta\sqrt a(t^\frac12-s^\frac12)
$$
for all $y\in B_{t}\left(x_0,r-\beta\sqrt{a t}\right)$.
\end{lma}

We also  need   the following expanding balls Lemma by He \cite{He2016}.

\begin{lma}\label{expanding-ball}
For  any positive integer $n\in \mathbb{N}$ and  for any  $v_0,\a, \sigma>0$, there exists   $\mu(n, v_0, \a, \sigma)>1$ and $R_0=R_0(n,v_0, \a,\sigma)>0$ such that the following is true: Let $(M^n,g(t))$ be a Ricci flow for   $t\in [0,T]$ with $T\le 1$.  Suppose $p\in M$ with  $B_0(p,R)\Subset M$ for some $R\ge R_0$ such that:

\begin{enumerate}
\item[(a)] $|\Rm(x,t)|\leq \a t^{-1}$ for all $x\in B_0(p,R)$ and $t\in(0,T]$;
\item[(b)] $V_t(x,\sqrt{t}) \geq \a^{-1} t^{n/2}$ for all  $t\in(0,T]$ and for all $x$ with   $B_t(x,\sqrt t)\subset  B_0(p,R)$;
\item[(c)] $ V_0(x,r)\leq v_0r^n$,  for all $0<r\le 1$ and $x\in B_0(p,R)$ with $ B_0(x,r)\subset B_0(p,R)$;
\item[(d)] $\tR_0\geq -\sigma$, in $ B_0(p,R)$, where $\tR_t$ is the scalar curvature of $g(t)$.
\end{enumerate}
  Then for all $t\in [0,T]$, we have $$B_0(p, \mu^{-1}R)\subset B_t(p,\frac12R).$$
\end{lma}
\begin{proof} By Lemma \ref{l-balls}, we have $B_t(p, \frac 34R)\Subset M$ for all $t\in [0,T]$ provided
 \vskip .1cm

{\bf (c1):}  $R\ge C_1$.

\vskip .1cm
Here and below, $C_l$ will denote a positive constant depending only on $n, v_0, \a, \sigma$.

By \cite[Lemma 8.1]{SimonTopping2016}, there is $T_1= T_1(n,\a,\sigma)>0$, such that
$\tR_t\ge -2\sigma$ on $B_t(p,\frac 23R)$  for $t\in [0,T\wedge T_1]$, if $C_1$ is large enough.

 Let $0<\tau\le T\wedge T_1\le 1$ and let $\beta=\beta(n)$ be the constant from Lemma \ref{l-balls}.  Let $R_1=\e R$ where $0<\e<1/2$ is a constant to be chosen later.   Define
$$r_0=\max\{r\in [0,R_1]: B_0(p,r)\Subset B_\tau(p,R_1/2)\}. $$

By Lemma \ref{l-balls} again, $r_0\le \frac12 R_1+\b\sqrt \a\le\frac12 R$, provided $C_1$ in {\bf(c1)} is  large enough. By definition,  there exists $y\in M$ such that $d_0(p,y)= r_0$ and $d_{\tau}(p,y)=R_1/2$. Let $\gamma:[0,r_0]$ be a minimizing $g(0)$-geodesic from $p$ to $y$. Let $N$ be the   positive integer such that
\be\label{e-N}
r_0+2\beta\tau^\frac12\ge  2\beta  \tau^\frac12N\ge r_0.
\ee

Then we can find  $\{x_i\}_{i=1}^N$  on $\gamma$ so that $B_0(x_i,\b\sqrt{\tau})
$ are all disjoint  and $\gamma$ is covered by $\bigcup_{i=1}^N B_0(x_i,2\b \sqrt{\tau})$ which is a subset of $ B_0(p,R)$ provided $C_1$ is large.
Choose $C_1$ large enough so that for each $i$, and for each $z\in B_0(x_i, 2\b\sqrt\tau) $ we have $B_\tau(z, 2\sqrt \tau)\subset B_0(p, R)$. For each $i$,     let $\{z_j^{(i)}\}_{j=1}^{k_i}$ be the maximal set of points in $B_0 (x_i,2\b  \sqrt{\tau})$ such that { $ B_\tau(z_j^{(i)},\sqrt{\tau})$ are mutually disjoint and
\begin{equation}
\begin{split}
\bigcup_{j=1}^{k_i} B_\tau(z_j^{(i)},\sqrt{\tau})&\subset B_0(x_i,2\b\sqrt{\tau})\subset \bigcup_{j=1}^{k_i} B_\tau(z_j^{(i)},2\sqrt{\tau})
\end{split}
\end{equation}}
Then  $\gamma$ will be covered by $\cup_{i=1}^N \cup_{j=1}^{k_i} B_\tau(z_j^{(i)},2\sqrt{\tau})$. Hence by \eqref{e-N}, we have
\be\label{e-expanding-1}
\frac{1}{2}R_1=d_\tau(p,y)\leq 2\sqrt \tau \sum_{i=1}^Nk_i.
\ee
We want to estimate $k_i$ from above.

Let  $\tau=\min\{T, T_1,  ({2\b})^{-2}\}$. By (b)   we have
\begin{equation}
\begin{split}
k_i \a^{-1}\tau^{n/2} &\leq  \sum_{j=1}^{k_i}V_\tau(z_j^{(i)},\sqrt{\tau})\\
&\leq V_\tau(B_0(x_i,2\b  \sqrt{\tau}))\\
&\leq C_2 V_0(B_0(x_i,2\beta \sqrt\tau))\\
&\leq C_2v_0\tau^\frac n2.
\end{split}
\end{equation}
The third inequality follows from $\partial_t d\mu_t=-\mathcal{R}_td\mu_t$ and the lower bound on $\mathcal{R}_t$. Hence $k_i\le C_3$.
By \eqref{e-expanding-1} and \eqref{e-N}, we have:
\bee
\begin{split}
\frac12 R_1
\le & 2C_3N\tau^\frac12 \leq 2C_3\tau^\frac12\cdot \frac{r_0+2\beta\tau^\frac12}{2\beta\tau^\frac12}
\end{split}
\eee
Therefore $r_0\geq C_4^{-1}R_1-2\b \tau^\frac12\ge C_4^{-1}R_1-1 $ and hence
\be\label{e-expanding-r1}
 B_0(p,\e C_4^{-1}R-1)\subset B_\tau(p,\frac{\e}{2}R).
 \ee

Suppose $ T\le T_1\wedge (2\beta)^{-2}$, then $\tau=T$. For all $t\le T=\tau$, By Lemma \ref{l-balls},
$$
B_\tau(p, \frac{\e}{2}R)\subset   B_t(p, \frac{\e}{2}R+\b\tau^\frac12)\subset  B_t(p, \frac{1}{2}R)
$$
provided $C_1$ in {\bf(c1)} is  large enough and $\e<\frac14$. If $T>T_1\wedge (2\beta)^{-2}$, then for $t\le \tau$, the above inequality is still true. For $T\ge t\ge \tau$,
   by condition (a), we have
$$
 B_\tau(p, \frac{\e}{2}R)\subset B_t(p, \e C_5 R)\subset B_t(p, \frac 14R  ).
$$
provided that we choose $ \e= \frac1{4(C_5+1)}$. By \eqref{e-expanding-r1} one can see that if $C_1$ is large enough, then the Lemma is true.
\end{proof}

\section{Local maximum principle Theorem \ref{t-MP1}}\label{Local-MP1}

 We are now ready to prove Theorem \ref{t-MP1}.
\begin{proof}[Proof of Theorem~\ref{t-MP1}]  Let $g(t)$ be a Ricci flow on $M^n\times[0,T]$ so that $\Ric(g(t))\le \a t^{-1}$ for some $\a>0$. Let $\varphi$ be a continuous function defined on $M\times[0,T]$ which satisfies \eqref{heat-equ-RF-1} in the   sense of barrier at those points where $\varphi>0$. Assume $\varphi\le 0$ at $t=0$.  Let $p\in M$ with $B_0(p,2)\Subset M$. We want to prove  that $\varphi(p,t)\le t^l$ for large $l$  provided $t\le T\wedge \wh T(n,\a,l)$. We will first show that $\varphi\le t^\frac12$ near $t=0$, and then we   improve the estimate  to higher powers of $t$.

 By replacing $L$ by its positive part if necessary, we may assume that $L\ge 0$.
By Lemma \ref{l-balls}, there is $T_1=T_1(n,\a)>0$ such that  $B_{t}(p,\frac32)\Subset M$ if $t\le T \wedge T_1$. Let $d_t(x)$ to be the distance function from $p$ with respect to $g(t)$. Then $d_t(x)$ is defined  on $B_t(p,1)$  and is realized by a minimizing geodesic from $p$. By \cite[Lemma 8.3]{Perelman2002},  there exists $c_1(n)>0$ such that
\be\label{e-Perelman}
(\frac{\p}{\p t}-\Delta_{g(t)})(d_t(x)+c_1\a\sqrt t)\ge0
\ee
for  $x\in B_t(p,1)\setminus B_t(p, \sqrt t)$ in the sense of barrier. 

 Let $\phi:[0,\infty)\to [0,1]$ be a smooth function such that
\be\label{e-cutoff-1}
\phi(s)=\left\{
  \begin{array}{ll}
  1&\hbox{\ for $0\le s\le \frac12$;}\\
      \exp(-\frac1{(1-s)})&\hbox{\ for $\frac34\le s\le 1$;} \\
    0, & \hbox{\ for $s\ge1$,}
  \end{array}
\right.
\ee
 and such that   $ \phi'\le 0,\;\phi''\ge -c\phi$ for some absolute constant $c>0$.

For $r\in [\frac{1}{2^4},1]$, let
$$ \Phi_r(x,t)=\exp(-cr^{-2}t)\cdot \phi\left(  \frac{d_t(x)+c_1\a\sqrt t}{r}\right).$$
Then   $\supp(\Phi_r(\cdot,t))\subset  B_t(p,r)$. Note that
$$
\frac r2\le d_t(x)+c_1\a\sqrt t\le r
$$
if and only if $r-c_1\a\sqrt t\ge d_t(x)\ge \frac r2-c_1\a\sqrt t$. Hence if $t\le T_1=\frac1{2^5}(c_1\a+1)^{-2}$, and $0<\phi(x,t)<1$, then $1\ge d_t(x)\ge \sqrt t$. Therefore, by \eqref{e-Perelman}, we have
\be\label{e-Phir}
\left(\frac{\partial}{\partial t} -\Delta_{g(t)}\right)\Phi_r\le0
\ee
in the sense of barrier   on $B_0(p, 2 )\times[0,T_1\wedge T]$.  We may also choose $T_1$ small enough so that $e^{-cr^{-2}T_1}\geq \frac{1}{2}$.   Let $\eta(t)\ge 0$ be a smooth function in $t$  such that $\eta(t)>0$ for $t>0$. We consider the function
$$
 F=-\Phi_r^m \varphi+\eta.
$$
Here $m$ is a positive integer to be  determined later.
 Assume $\eta$ is chosen so that $F>0$ near $t=0$.  In this case,  if
 $F(x,t)<0$ for some $(x,t)\in B_0(p, 2)\times[0,T_1]$ then there is $(x_0, t_0)\in B_0(p, 2)\times(0,T_1]$ with $0<t_0\le T_1$ such that  $
F(x_0,t_0)=0$, and $F(x,t)\ge0$ on $ B_0(p, 2)\times[0, t_0].$  At $(x_0,t_0)$,  $\Phi_r>0$ and $ \varphi>0$.

 By \eqref{e-Phir} and \eqref{heat-equ-RF-1},  for any $\e>0$, there exists $C^2$ functions $\sigma(x)$, $\zeta(x)$ near $x_0$ such that $\sigma(x)\le \Phi_r(x,t_0)$, $\sigma(x_0)=\Phi_r(x_0,t_0)$, $\zeta(x)\le \varphi(x,t_0)$ and $\zeta(x_0)=\varphi(x_0,t_0)$. Moreover,  the following are true:
$$
\frac{\p_-}{\p t}\Phi_r(x_0,t_0)-\Delta_{g(t)}\sigma(x)\le \e,
$$
and
$$
\frac{\p_-}{\p t}\varphi(x_0,t_0)-\Delta_{g(t)}\zeta(x_0)-L(x_0,t_0)\zeta(x_0)\le\e.
$$
 Here for a function $f(x,t)$,
$$
\frac{\p_-}{\p t}f(x_0,t_0)=\liminf_{h\to 0^+}\frac{f(x_0,t_0)-f(x_0,t_0-h)}{h}.
$$

The function $ G(x,t)=-\sigma^m(x)\zeta(x)+\eta(t)$  is $C^2$ in space and time so that $G(x_0,t_0)=F(x_0,t_0)=0$.  For $x$ near $x_0$, since $\varphi(x,t_0)>0$ near $x_0$ and $\varphi$ is continuous, for $x$ sufficiently close to $x_0$,
\bee
G(x,t_0)\ge -\Phi_r^m(x,t_0)\varphi(x,t_0)+\eta(t_0)=F(x,t_0)\ge0.
\eee
 Hence at $(x_0,t_0)$, we have
\bee
\left\{
  \begin{array}{ll}
    \zeta={\displaystyle\frac{\eta}{\Phi_r^m}};  \\
    \nabla \zeta={\displaystyle-\frac{m\zeta\nabla \sigma}{\sigma};}\\
|\nabla\sigma|\le |\nabla\Phi_r|,
  \end{array}
\right.
\eee
and
\bee
\begin{split}
0\le &\Delta_{g(t)} G\\
=&-\sigma^m\Delta_{g(t)} \zeta-\zeta \Delta_{g(t)}\sigma^m-2\la\nabla\sigma^m,\nabla\zeta \ra\\
=&-\sigma^m\Delta_{g(t)} \zeta-m \zeta\sigma^{m-1}\Delta_{g(t)} \sigma -m(m-1)\zeta\sigma^{m-2}|\nabla \sigma|^2-2m\sigma^{m-1}\la \nabla\sigma,\nabla \zeta\ra\\
\le&\Phi_r^m\lf(-\frac{\p_-}{\p t}\varphi  +L\varphi+\e\ri)
+m\Phi_r^{m-1}\varphi\lf(-\frac{\p_-}{\p t}\Phi_r(x_0,t_0)+\e\ri)-2m\Phi_r^{m-1}\la \nabla\sigma,\nabla \zeta\ra\\
=&\frac{\p_-}{\p t}F-\eta'+\Phi_r^m\lf(  L\varphi+\e\ri)+\e n  \Phi_r^{m-1}\varphi +2m^2\sigma^{m-2}\zeta|\nabla\sigma|^2\\
\le &-\eta'+L\eta+\e\Phi_r^m +\e m  \Phi_r^{m-1}\varphi+ 2m^2\eta\frac{|\nabla \Phi_r|^2}{\Phi^2}, \\
\end{split}
\eee
 because
\bee
\frac{\p_-}{\p t}F\Big|_{(x_0,t_0)}\le 0.
\eee

 By letting $\e\to0$, we conclude that at $(x_0,t_0)$,
\be\label{e-eta-1}
\begin{split}
\eta'(t_0)\le& \eta(t_0)\lf( L(x_0,t_0)+ 2m^2\frac{|\nabla\Phi_r|^2}{\Phi_r^2(t_0)}\ri)\\
\le &\left\{
       \begin{array}{ll}
        \eta(t_0)\lf(L_0 + C_1^2m^2\displaystyle{\lf(\frac{a_0}{\eta(t_0)}\ri)^\frac2m} \ri) ; \text{\rm\ \  or} \\
          \eta(t_0)\lf(\displaystyle{\frac{\a}{t_0}+C_1^2m^2\lf(\frac\a{t_0\eta(t_0)}\ri)^\frac2m} \ri),
       \end{array}
     \right.
\end{split}
\ee
where $ L_0=\max_{B_0(p, 2 )\times[0,T]} L$, $ a_0=\max_{B_0(p, 2)\times[0,T]} |\varphi|$.
 In the above inequalities,  we have used the fact that   at $(x_0,t_0)$,
$$
\frac1{\Phi_r^m}=\frac{\varphi}{\eta}\le \min\left\{\frac \a{t_0\eta(t_0)},\frac{a_0}{\eta(t_0)}\right\}.
$$
Here and below $C_k$ will denote a positive constant depending only on $n, \a$.

First, we show that $\varphi(t)=O(t^{1/2})$. For any  $1>\delta>0$, let $\eta(t)=t^\frac12+\delta$.  Then $F>0$ near $t=0$. By the first inequality  on the second line of \eqref{e-eta-1}, we have
$$
\frac12 t_0^{-\frac12}\le (t_0^\frac12+\delta)\lf(L_0+\frac{C_1^2m^2 a_0^\frac2m}{(t_0^\frac12+\delta)^\frac2m}\ri).
$$

Choose $m=2$ and $r=1$,   we see that there is $\tau>0$, small enough but independent of $\delta$ so that  $t_0\ge \tau$. Hence by letting $\delta\to0$,  we conclude that $ \varphi\le 2t^\frac12$ on $B_{t}(p,\frac12-c_1\a \sqrt{t})$ near $t=0$.

Next we improve the estimate  of $\varphi$ as $t\to0^+$. 
Given an integer $k\ge1$ and $\delta>0$, let $\eta=\delta t^\frac14+t^k$ and $r=\frac{1}{2}$. By the first inequality on the second line of \eqref{e-eta-1}, we have
\bee
\frac14\delta t_0^{-\frac14}+kt_0^{k-1}\le (\delta t_0^\frac14+t_0^k) \lf(L_0+\frac{C_1^2m^2 a_0^\frac2m}{(\delta t_0^\frac14+t^k_0)^\frac2m}\ri) .
\eee
Choose $ m$ large enough so that $2k/m<1$, then we can find $\tau_1>0$ such that $t_0>\tau_1$. Therefore, we may conclude that $\varphi(x,t)\le 2t^k$ near $t=0$ on $B_0(p,\frac14-c_1\a\sqrt{t})$.

Now we will show that under \eqref{e-assumption-1}, for each $l\geq \a+1$, the above $\tau_1$ can be chosen so that it is bounded from below away from zero depending only on $n, \a,l$. Let $\eta= \frac{1}{2} t^l $, $r=\frac{1}{4}$. By the  above upper estimate of $\varphi$ near $t=0$,
we see that $F>0$ near $t=0$. Therefore, we can use the second inequality on the second line of \eqref{e-eta-1} to show that
\bee
 lt_0^{l-1}\le   t_0^l  \lf(\frac{\a}{t_0}+C_1^2m^2\lf(\frac{\a}{t_0^{l+1}}\ri)^{\frac 2m}\ri).
\eee
This implies:
\bee
  t_0^{l-1} \le C_1^2m^2\a^\frac2m t_0^{l-\frac2m(l+1)}.
\eee
Choose $m$ sufficiently large so that $\frac2m(l+1)<\frac12$, we conclude that $t_0\ge T_2(n,\a, l)$  and hence
$$
\varphi(p,t)\le t^l
$$
if $t\le T_2\wedge T_1\wedge T$. This completes the proof.
\end{proof}

Theorem \ref{t-MP1} is invariant under parabolic rescaling  in the following sense: Let $(M,g(t))$, $\varphi, L$ be as in the theorem. For any  $\lambda>0$,  we define $g_1(x,t)=\lambda g(x,\lambda^{-1}t)$, $\varphi_1(x,t)=\lambda^{-1} \varphi(x,\lambda^{-1}  t)$ and $L_1(x,t)=\lambda^{-1}  L(x,\lambda^{-1}t)$. Then $g_1(t)$ satisfies the curvature condition \eqref{e-assumption-1} with the same $\a$ and
{ $$\varphi_1(x,t)=\lambda^{-1} \varphi(x,\lambda^{-1} t)\le \lambda^{-1} \a (\lambda^{-1} t)^{-1}=\a t^{-1}.$$}
Similarly,  $L_1(x,t)\le \a t^{-1}$.
Moreover, let $s=\lambda^{-1} t$
\bee
\begin{split}
\lf(\frac{\p}{\p t}-\Delta_{g_1(t)}\ri)\varphi_1{ (x,t)} =& \lambda^{-2}\lf(\frac{\p}{\p s}-\Delta_{g(s)}\ri)\varphi{(x,s)}\\
\le&\lambda^{-2} L(x,s)\varphi{ (x,s)}\\
=&L_1(x,t)\varphi_1{ (x,t)}
\end{split}
\eee
in the sense of barrier whenever $\varphi_1(x,t)=\lambda^{-1}\varphi(x,\lambda^{-1}t)>0$. Hence we have the following rescaled version of Theorem~\ref{t-MP1}.
\begin{cor}\label{c-MP1-1}
Let  $(M,g(t)),t\in [0,T]$, $\varphi, L$ be as in the Theorem \ref{t-MP1}.  Let $p\in M$, $r>0$ with $B_0(p,r)\Subset M$. Then for any positive integer $l\ge \a+1$, there is $T_1(n,\a, l){ >0}$ depending only on $n, \a, l$ such that
 $$
\varphi(p,t)\le     r^{-2(l+1)}t^l \
$$
for all $t\le [0, T\wedge r^{2}T_1]$.
\end{cor}
\begin{proof} Let $\lambda=r^{-2} 4$. Define $g_1, \varphi_1, L_1$ as above. Then $B_{g_1(0)}(p, 2)\Subset M$. By Theorem \ref{t-MP1},  for any  $l\ge \a+1$
$$
\varphi_1(p,t)\le  t^l
$$
for $t\in [0,T_1'\wedge \lambda T]$ for some $T_1' >0$,   depending only on $n, \a, l$. Hence
\bee
\begin{split}
\varphi(p,t)=&\lambda \varphi_1(p,\lambda t)\\
\le & \lambda^{ l+1}t^{l}\\
=&r^{-2(l+1)}4^{-(l+1)}t^{l+1}\\
\le&r^{-2(l+1)}t^l
\end{split}
\eee
for $t\in [0, (r^2T_1)\wedge T]$ where $T_1=4T_1'$. From this, the result follows.
\end{proof}

When $g(t)$ is  a complete solution to the Ricci flow, then the corollary implies that $\varphi\le 0$ for $t>0$ by letting $r\to\infty$. In fact, in this case by using the trick of Chen \cite{Chen2009} we do not need the assumption that $\Ric(g(t))\le \a t^{-1}$. Namely we have following corollary of our method:

\begin{cor}\label{c-MP1-2}
Let $(M^n,g(t))$ be a complete solution of the Ricci flow  on $M\times[0,T]$. Let  $\varphi, L$ be as in Theorem \ref{t-MP1}.   Then $\varphi\le 0$ for $t>0$.
\end{cor}
\begin{proof}   For any compact set $\Omega$, we have $\Ric(g(t))\le   t^{-1}$ on $\Omega$ provided $t$ is small enough. Hence by Corollary \ref{c-MP1-1}, for any $l\ge1$ and compact set $\Omega$, we have $\varphi\le t^l$ on $\Omega$ provided $t$ is small depending only on $n,l$ and $\Omega$.

Let $p\in M$ and let $b$ be a positive number such that $\Ric(g(t))\le b^2$ on $B_t(p,1)$ for all $t\in [0,T]$. Then as before
\begin{equation}\label{d-evo--1}
\left(\frac{\partial }{\partial t}-\Delta_{g(t)}\right)(d_t(x)+c_1bt)\ge0
\end{equation}
in the sense of barrier on $M\setminus B_t(p, \frac1b)$ for some $c_1=c_1(n)$. Let $\phi$ be as in \eqref{e-cutoff-1}. Then for $A>1$ sufficiently large,
$$
\left(\frac{\partial }{\partial t}-\Delta_{g(t)}\right)\phi\left(\frac{d_t(x)+c_1bt}{A}\right)\le  \frac{c}{A^2}\phi\left(\frac{d_t+c_1bt}{A}\right)
$$
in the sense of barrier.
Define $\Phi(x,t)=\exp(-\frac{c}{A^2}t)\phi$ so that
\bee
\left(\frac{\partial }{\partial t}-\Delta_{g(t)}\right) \Phi\le 0.
\eee
in the sense of barrier.   For any integer $m\ge 2$, $l>1+\a$ and $\e>0$, we define $F=-\Phi^m \varphi+\varepsilon t^l$. The  argument in the beginning of the proof of Theorem \ref{t-MP1} shows that $F(x,t)> 0$ on $M\times(0,t')$ for some $t'>0$. If $F(x,t)<0$ somewhere, then there is $x_0\in M$, $T\ge t_0>0$ so that $F(x_0,t_0)=0$ and $F(x,t)\ge 0$ on $M\times[0,t_0]$. As in  \eqref{e-eta-1}, we have
$$
\varepsilon lt_0^{l-1}\le \varepsilon t_0^l\lf(\displaystyle{\frac\a{t_0}+ \frac{C_1^2m^2}{A^2}
\lf(\frac{\a} { \varepsilon t_0^{l+1}}\ri)^\frac 2m}\ri).
$$
And hence,
$$
  \varepsilon^\frac 2m A^2 t_0^{l-1}\le  t_0^{l-\frac2m(l+1)}\a^\frac2m C_1^2m^2,
$$
where we have used the fact that $\ell>\a+1$.

For a fixed $l$, we choose $m$ sufficiently large such that $\frac2m(l+1)<\frac12$, then we have
$$
\varepsilon^\frac 2m A^2 C_2\le t_0^{1-\frac2m(l+1)}.
$$
Therefore, if we choose $A$ large enough so that $\varepsilon^\frac 2m A^2 C_2>T^{1-\frac2m(l+1)}$. We conclude that $t_0>T$ which is impossible. By letting $A\rightarrow +\infty$ and followed by $\varepsilon\to0$, we conclude that $\varphi(p,t)\le0$. Since $p$ is an arbitrary point on $M$, the result follows.
\end{proof}

\begin{cor}Let $(M^n,g)$ be a complete Riemannian manifold with $ \Ric(g)\geq -1$. Let  $\varphi, L$ be as in Theorem \ref{t-MP1} with respect to $\partial_t-\Delta_g$ instead.   Then $\varphi\le 0$ for $t>0$.
\end{cor}
\begin{proof}
By Laplacian comparison, we have $(\partial_t-\Delta_g)(d_{g}(x,p)+Ct)\geq 0$ in the sense of barrier for some fixed $p\in M$ and $C>0$ whenever $d_g(x,p)>1$. Hence, the proof of Corollary \ref{c-MP1-2} can be carried over.
\end{proof}

 Using the idea in \cite{Shi1997}, we may use
Corollary \ref{c-MP1-2} to prove that complete Ricci flows satisfying curvature condition $|\Rm(g(t))|\le \a t^{-1}$ preserve the \K condition. This recovers   results in \cite{Shi1997,HuangTam2018}.  Another application of Corollary \ref{c-MP1-2} is on the preservation of non-negativity of various curvatures  along the Ricci flows   which may not be complete or may have  unbounded curvature. We will follow the set-up in \cite{BCRW}.
 See \cite{Wilking2013} for a unified approach in compact case and the case that $g(t)$ is complete noncompact with bounded curvature. Information about previous contributions by others can also be found in \cite{Wilking2013}.

\begin{thm}\label{t-preservation}
Let $(M^n,g(t))$ be a smooth solution to the Ricci flow on $M\times[0,T]$,  where $g(t)$ may not be complete. Assume the scalar curvature $\tR$ satisfies $\tR(g(t))\le \a t^{-1}$ for some $\a>0$ on $M\times(0,T]$. Consider one of the following curvature conditions $\mathcal{C}$:

\begin{enumerate}

\item   non-negative curvature operator;
\item   2-non-negative curvature operator,
(i.e. the sum of the lowest two eigenvalues is non-negative);
\item  { weakly $\mathrm{PIC}_2$ (i.e. taking the Cartesian product with $\R^2$ produces
a non-negative isotropic curvature operator);}
\item  weakly $\mathrm{PIC}_1$
(i.e. taking the Cartesian product with $\R$ produces a non-negative isotropic
curvature operator);
\item  non-negative bisectional curvature, in the case in which $(M, g(t))$ is K\"ahler;
\item non-negative orthogonal bisectional curvature, in the case in which $(M, g(t))$ is \K.
      \end{enumerate}

      Let $p\in M$ and $r>0$ be such that $B_0(p,r)\Subset M$,    $\Rm(g(x,0))\in \mathcal{C}$ and $\Rm(g(t))+\a t^{-1}\mathrm{I}$ is in the same $\mathcal{C}$ for $(x,t)\in B_{0}(p,r)\times (0,T]$. Then for all $l>\a+1$, there is $\wh T(n,\a,l) >0$ such that for all $t\in [0,T\wedge \wh T r^2]$, $\Rm(g(p,t))+r^{-2(l+1)}t^l\mathrm{I}$ is in the same $\mathcal{C}$. In particular, if $g(t)$ is a complete solution and the assumption holds for all $r>0$, then $\Rm(g(t))\in \mathcal{C}$ for all $t>0$.
\end{thm}
\begin{proof} Fix a curvature condition $\mathcal{C}$.
Let $$\ell(x,t)=\inf\{\varepsilon>0|\ \Rm(g(x,t))+\varepsilon I\in \mathcal{C}\}. $$ Then by \cite{BCRW} for (1)--(5) and by \cite{NiLi2019} for (6), we have
$$
\left(\frac{\partial}{\partial t}-\Delta_{g(t)} \right)\ell\le \tR\ell+c(n)\ell^2
$$
in the sense of barrier for some constant $c(n)$ depending only on $n$. By the assumption on $\tR$ and the assumption that $\Rm(g(t))+\a t^{-1}I\in \mathcal{C}$, we conclude that
$$
\left(\frac{\partial}{\partial t}-\Delta_{g(t)} \right) \ell \le at^{-1}\ell
$$
for some $a>0$. { Since} $\ell=0$ at $t=0$ { as} $ \Rm(g(0))\in \mathcal{C}$, the conclusion at $p$ follows from Corollary \ref{c-MP1-1}. When $g(t)$ is a complete solution, we can let $r\rightarrow +\infty$ to conclude that $\Rm(g(p,t))\in\mathcal{C}$. Since $p$ is arbitrary, the result follows.
\end{proof}

\section{Local maximum principle Theorem \ref{t-MP2}}\label{Local-MP2}

  In this section, we will use Theorem \ref{t-MP1} to prove Theorem~\ref{t-MP2}.
    Let  $g(t)$ be a Ricci flow on $M\times[0,T]$ satisfying:
\be\label{e-assumptions}
 \left\{
   \begin{array}{ll}
     |\Rm(g(x,t))|\leq \a t^{-1},& \hbox{for all $(x,t)\in M\times(0, T]$}\\
     \mathrm{inj}_{g(t)}(x) \geq \sqrt{\a^{-1}t},&\hbox{for all $(x,t)\in M\times(0, T]$ with $B_t(x,\sqrt{\a^{-1}t})\Subset M$},
   \end{array}
 \right.
\ee
for some $\a>1$. We will consider the continuous function $\varphi(x,t)\ge 0$   on $M\times[0,T]$ which satisfies:
\be\label{e-phi-1}
\left(\frac{\partial }{\partial t}-\Delta_{g(t)}\right) \varphi\leq \tR\varphi+K\varphi^2
\ee
in the sense of barrier,  where $\tR(g(t))$ is the scalar curvature of $g(t)$ and $K\ge0$ is a constant.

 Before we prove Theorem \ref{t-MP2}, we first give the following application of the theorem.
\begin{cor}\label{c-t-MP2} Let $(M^n,g(t))$ be as in Theorem \ref{t-MP2}. Suppose $$\Rm(g_0)+\delta I\in \mathcal{C}$$
for some $\delta>0$ where $\mathcal{C}$ is one of the curvature cones (1)--(6) in Theorem \ref{t-preservation}. Let $p, R$ be as in Theorem \ref{t-MP2}.  Then  there is a constant $C_0=C_0(n,\a, v_0,\sigma)>0$ such that
$$\Rm(g(p,t))+\delta^*I\in \mathcal{C} $$
for $t\in [0,T]$ where $\delta^*=C_0((RT^\frac12)^{-2}+\delta)$.
In particular, if $g(t)$ is a complete solution,
then $\Rm(g(t))+C_0\delta\in \mathcal{C}$.
\end{cor}
\begin{proof} As in the proof of Theorem \ref{t-preservation}, let $$\ell(x,t)=\inf\{\varepsilon>0|\ \Rm(g(x,t))+\varepsilon I\in \mathcal{C}\}.$$ Then
$$
\left(\frac{\partial}{\partial t}-\Delta_{g(t)} \right) \ell\le \tR\ell+c(n)\ell^2
$$
in the sense of barrier for some constant $c(n)$ depending only on $n$.  By   \eqref{e-assumptions}, $\ell(x,t)\le \a't^{-1}$ for some $\a'>0$ depending only on $\a, n$.   The first result follows from the Theorem~\ref{t-MP2}. The second result follows by letting $R\to\infty$.
\end{proof}

We may  reduce the proof of Theorem~\ref{t-MP2} to the case that $T=1$. More precisely, let $g_1(t)=T^{-1}g(Tt)$, $\varphi_1(x,t)=T\varphi(x,Tt)$. Then $ g_1(t)$ satisfies  \eqref{e-assumptions} and $\varphi_1(x,t)$ satisfies \eqref{e-phi-1} with $\tR(g(t))$ replaced by $\tR(g_1(t))$ .  Moreover, {$g_1(0)$ and $\varphi_1$ satisfy}
\be\label{e-volume}
  \left\{
    \begin{array}{ll}
       \tR(g_1(0))\ge -\sigma     & \hbox{on $M$;} \\
      V_{g_1(0)}(x,r)\le  r^n\exp(v_0 r)  & \hbox{for all $r>0$, and  $x\in M$ with $B_0(x,r)\Subset M$,}
    \end{array}
  \right.
\ee
  and
\bee
\left\{
  \begin{array}{ll}
    \varphi_1(0)\leq T\delta, & \hbox{on $M$;} \\
    \varphi_1(t)\leq \a t^{-1}, & \hbox{on $M\times(0,T]$.}
  \end{array}
\right.
\eee
 If we can prove Theorem \ref{t-MP2} with $T=1$, then  the upper bound for $\varphi_1$ will imply   the required upper bound for $\varphi$.

 Theorem \ref{t-MP2} has been obtained earlier by Hochard \cite[Proposition I.2.1 \& Proposition II.2.6]{Hochard2016}. 
 The approach here is a localized version of the original method in \cite{BCRW}. The main ingredients are the local maximum principle Theorem \ref{t-MP1} and an upper bound of the Dirichlet heat kernel for a fixed domain. 

\subsection{Upper estimates of the Dirichlet heat kernel}\label{s-HK-estimates}
Let $(M^n,g_0)$ be a Riemannian manifold and let $g(t)$ be a Ricci flow on $M\times[0,T]$ with $g(0)=g_0$. Here $g(t)$  may   be incomplete.    Let $\Omega \Subset M $ be an open set with smooth boundary.  We let $G_\Omega(x,t;y,s)$, $t>s$ be the Dirichlet heat kernel for the backward heat equation coupled with the Ricci flow  $g(t)$:
     \be\label{e-kernel-2}
\left\{
  \begin{array}{ll}
  (\partial_s +\Delta_{y,s})\, G_\Omega(x,t;y,s)=0, & \text{in $ int (\Omega) \times int(\Omega)\times[0,t)$}; \\
     \lim_{s\rightarrow t^-}G_\Omega(x,t;y,s) =\delta_{x}(y), & \text{for $x\in int(\Omega)$};\\
     G_\Omega(x,t;y,s) =0, & \text{for $y\in \partial \Omega$, $x\in int(\Omega)$,}
  \end{array}
\right.
\ee
where $\Delta_{y,g(s)}$ is denoted by $\Delta_{y,s}$. Then
\be\label{e-kernel-1}
\left\{
  \begin{array}{ll}
  (\partial_t -\Delta_{x,t} -\tR_t)\, G_\Omega(x,t;y,s)=0, & \text{in $ int (\Omega) \times int(\Omega)\times(s,T]$}; \\
     \lim_{t\rightarrow s^+}G_\Omega(x,t;y,s) =\delta_{y}(x), & \text{for $y\in int(\Omega)$};\\
     G_\Omega(x,t;y,s) =0, & \text{for $x\in \partial \Omega$, $y\in int(\Omega)$,}
  \end{array}
\right.
\ee
where $\tR_t$ is the scalar curvature of $g(t)$.
 Such $G_\Omega$ exists and  is positive in the interior of $\Omega$, see \cite{Guenther}.

We want to estimate the   upper bound  of $G_\Omega(x,t;y,s)$ with respect to $y$ and $g(s)$  under the conditions \eqref{e-assumptions}. The following Dirichlet  heat kernel estimate was implicitly proved in \cite[Theorem 5.1]{ChauTamYu2011}.

\begin{lma}\label{l-Heat-Kernel-est}
Let $(M^n,g_0)$ be a Riemannian manifold and $p\in M$. Suppose $g(t)$ is a solution to the Ricci flow on $M\times [0,1]$ with $g(0)=g_0$ such that $B_{0}(p,{  2(r+1)})\Subset M$ for some $r\geq 1$ and $|\Rm(x,t)|\leq A$ on $M\times [0,1]$. If $\Omega$ is an open set in $M$ with smooth boundary such that $\Omega \Subset B_{g_0}(p,r)$ and $G_\Omega(x,t;y,s)$ is the Dirichlet heat with respect to the backward heat equation on $\Omega\times \Omega\times [0,1]$. Then there is $C(n,A)>0$ such that for all $0\leq s<t\leq 1$, $x,y\in\Omega$,
$$
G_\Omega(x,t;y,s)\le
\frac{C }{V^\frac{1}{2}_{0}(x,\sqrt{t-s})V^\frac{1}{2}_{0}(y,\sqrt{t-s})}
\times\exp\left(-\frac{d^2_{0}(x,y)}{C (t-s)}\right).
$$
\end{lma}

\begin{proof} Let $\wt\Omega$ be a bounded open domain with smooth boundary so that $B_{0}(p,r+\frac12)\Subset \wt\Omega\Subset B_{0}(p,r+1)$. In particular, any two points in $\wt\Omega$ can be joined by a minimizing geodesic in $M$. Let $G_{\wt\Omega}$ be the heat kernel on $\wt\Omega\times\wt\Omega\times[0,1]$. By the maximum principle, we have $G_\Omega\le G_{\wt\Omega}$ on $\Omega\times\Omega\times[0,1]$. In the following $C_i$ will denote positive constants depending only on $n, A$.

 \textsc{Step 1}: Denote $G_{\wt \Omega}$ by $G$. For $0<s<t\le 1$,
\bee
\begin{split}
 \frac{\p}{\p t} \left(\int_{\wt\Omega} G(x,t;y,s) d\mu_{x,t}\right)&=\int_{\wt\Omega} \Delta_{x,t} G\; d\mu_{x,t}=\int_{\partial\wt \Omega} \frac{\partial G}{\partial \nu} \le 0,
\end{split}
\eee
because $G> 0$ on $int(\tilde \Omega)$ and $G=0$ on $\partial \tilde\Omega$. Since
 $\lim_{t\rightarrow s^+ }\int_{\wt\Omega} G(x,t;y,s)d\mu_{x,t}=1,$
we have  for all $y\in \wt\Omega$.
\be\label{e-hk-1}
\int_{\wt\Omega}G_{\wt \Omega}(x,t;y,s) d\mu_{x,t}\le 1.
\ee

Let $f\in C^\infty(\tilde\Omega)$ with $0\leq  f\leq 1$ on $\tilde\Omega$ and $f=0$ on $\partial\tilde \Omega$. Let
$$
u(x,t)=\int_{\wt \Omega} G_{\wt\Omega}(x,t;y,s) f(y) \;d\mu_{y,s}.
$$
Then $u$ satisfies $ (\partial_t-\Delta_{g(t)}-\tR_t)u=0$ with zero boundary data and with initial data $f$.  By the maximum principle, we have $u(x,t)\leq C_1$ for $t\geq s$. Letting $f\rightarrow 1$, we conclude that
\be\label{e-hk-2}
\int_{\wt \Omega} G_{\wt\Omega}(x,t;y,s) \;d\mu_{y,s}\le C_1.
\ee

\textsc{Step 2}: Apply the argument of \cite[Lemma 5.3]{ChauTamYu2011} and {\sc Step 1}, using the mean value inequality \cite[Lemma 3.1]{ChauTamYu2011}, volume comparison and the fact that $B_{g_0}(x,\frac12)\Subset \wt\Omega$ for all $x\in \Omega$, we have pointwise estimate:
\bee
G_\Omega(x,t;y,s)\leq  G_{\wt \Omega}(x,t;y,s)\leq \min\left\{\frac{C_2}{V_{0}(x,\sqrt{t-s})},\frac{C_2}{V_{0}(y,\sqrt{t-s})}\right\}.
\eee
Combining this  with the integral estimates in {\sc Step 1},  we conclude that for $y\in \Omega$ and $s<t$,
\begin{equation}
\left\{
\begin{array}{ll}
\displaystyle\int_\Omega G_\Omega^2(x,t;y,s)\;d\mu_{x,t}\leq \frac{C_3}{V_{0}(y,\sqrt{t-s})};\\
 \displaystyle\int_\Omega G_\Omega^2(x,t;y,s)\;d\mu_{y,s}\leq \frac{C_3}{V_{0}(x,\sqrt{t-s})}.
\end{array}
\right.
\end{equation}

\textsc{Step 3}: Apply  the method of proof in  \cite[Theorem 2.1]{Grigoryan1997}), (see also \cite[Lemma 2.2]{ChauTamYu2011}), we have
\begin{equation}\label{e-est-1}
\left\{
\begin{array}{ll}
\displaystyle \int_\Omega G_\Omega^2(x,t;y,s)e^\frac{d_{0}^2(x,y)}{C_4(t-s)}\;d\mu_{x,t}&\displaystyle{\leq \frac{C_4}{V_{0}(y,\sqrt{t-s})} \quad\text{for all $y\in \Omega$; and}}\\
\displaystyle\int_\Omega G_\Omega^2(x,t;y,s)e^\frac{d_{0}^2(x,y)}{C_4(t-s)}\;d\mu_{y,s}&\leq \displaystyle{ \frac{C_4}{V_{0}(x,\sqrt{t-s})}\quad \text{for all $x\in \Omega$}}.
\end{array}
\right.
\end{equation}

{\sc Step 4}: By the semi-group property of Dirichlet heat kernel (see \cite[Lemma 26.12]{ChowBookIII} for example), and by using arguments in the proof of \cite[Theorem 5.5]{ChauTamYu2011}, we have
$$G_{\Omega}(x,t;y,s)\leq \frac{C_5}{V^\frac{1}{2}_{0}(x,\sqrt{t-s})V^\frac{1}{2}_{0}(y,\sqrt{t-s})}
\times\exp\left(-\frac{d^2_{0}(x,y)}{C_5(t-s)}\right).$$
\end{proof}

  Using Lemma~\ref{l-Heat-Kernel-est},  we can now proceed as in   \cite[Proposition 3.1]{BCRW} to obtain the following heat kernel estimate.
\begin{prop}\label{l-Green-esti}
For any $n,\a>0$, there exists $C(n,\a)>0$ such that the following is true: Suppose $(M^n,g(t))$ is a solution to the Ricci flow on $M\times[0,1]$ with initial metric $g_0$ satisfying the conditions \eqref{e-assumptions}.  Let $p\in M$ be a fixed point so that $B_t(p,4r)\Subset M$ for some $r\ge 1$ for all $t\in [0,1]$. Let $\Omega$ be  a domain  with smooth boundary so that $ \Omega\Subset B_t(p,r)$ for all $t\in[0,1]$.
Then  the Dirichlet heat kernel $G(x,t;y,s)$    with respect to the backward heat equation on $\Omega\times\Omega\times[0,1]$ satisfies:
\begin{equation}
 G(x,t;y,s)\leq \frac{C}{(t-s)^{\frac n2}}\exp\left(-\frac{d^2_s(x,y)}{C(t-s)} \right).
\end{equation}
for all $0\leq  s<t\leq 1$ and $x,y\in \Omega$, where $d_s$ is the distance function with respect to $g(s)$.
\end{prop}

\begin{rem}
The condition that $\Omega\Subset B_t(p,r)\subset B_t(p,4r)\Subset M$ is assumed so that for all $x,y\in \Omega$, $d_t(x,y)$ is well-defined and is realized by a  minimizing $g(t)$-geodesic lying in $B_t(p,4r)$.
\end{rem}

By Lemma \ref{expanding-ball}, we have the following:
\begin{cor}\label{c-rescaled-heat-esti} There exist  $R_0(n,\a,\sigma,v_0)>1,\mu(n,\a,\sigma,v_0)>1$ such that the following is true:
Suppose $(M^n,g_0)$ is a Riemannian manifold and $g(t)$ is a solution to the Ricci flow on $M\times [0,T]$  with $g(0)=g_0$ satisfying conditions \eqref{e-assumptions} and
\bee
\left\{
  \begin{array}{ll}
    \tR(g_0)\ge -\sigma T^{-1}  &\hbox{on $M$;} \\
    V_0(x,r)\le v_0r^n   & \hbox{for $0<r\le T^{\frac12}$.}
  \end{array}
\right.
\eee
Then we can find $C(n,\a)>0$ so that if $p\in M$ with $B_0(p,R)\Subset M$ and $R\ge T^\frac12R_0$, then   the heat kernel  $G(x,t;y,s)$ on $B_{g_0}(p,\mu^{-1}R)\times [0,T]$  satisfies
 \begin{equation}
 G(x,t;y,s)\leq \frac{C}{(t-s)^{\frac n2}}\exp\left(-\frac{d^2_s(x,y)}{C(t-s)} \right).
\end{equation}
\end{cor}
\begin{proof} Let $ g_1(t)=T^{-1}g(Tt)$. Then $ g_1(t)$ is a Ricci flow defined on $M\times[0,1]$ satisfying \eqref{e-assumptions} and \eqref{e-volume}. By Lemma \ref{expanding-ball} and Proposition \ref{l-Green-esti}, the results follows by parabolic rescaling.
\end{proof}

\subsection{Proof of Theorem \ref{t-MP2}}
The following Lemma reduces the upper bound of $\varphi$ to the integral bound of the heat kernel.
\begin{lma}\label{l-mp-1}
Suppose $(M^n,g(t)),t\in [0,T]$ is  a smooth solution to the Ricci flow with initial metric $g_0$ which may not be complete. Suppose $g(t)$ satisfies:
\bee
    \Ric (x,t)\leq \frac{\a}{t}
\eee
 for $(x,t)\in M\times(0, T]$ for some $\a>0$.  Let $\varphi$ be a nonnegative continuous function on $M\times [0,T]$ such that   $\varphi(0)\leq \delta$ and $\varphi(t)\leq \a t^{-1}$ for some $\delta >0$.  Assume $\varphi$ satisfies
$$\left(\frac{\partial}{\partial t}-\Delta_{g(t)} \right) \varphi\leq \tR\varphi+K\varphi^2
$$
in the sense of barrier, where   $K>0$ is a constant and    $\tR=\tR(g(t))$ is the scalar curvature of $g(t)$. Let $p\in M$ such that $B_{g_0}(p,r)\Subset M$. Then there exists
  constants $C(n,\a, K), T_1(n,\a,K)>0$ so that
\begin{equation}\label{e-quantitative}
\varphi(p,t)\leq
C\lf(r^{-2}+\delta \mathcal{S} \ri)
\end{equation}
for all $t\in [0, T\wedge r^2T_1]$, where
$$
\mathcal{S}=\sup_{B_0(p,r)\times  [0, r^2T_1\wedge T]}\int_{B_{g_0} (p, r)}G(x,t;y,0)d\mu_{y,0},
$$
and   $G(x,t;y,s)$ is the heat kernel for the backward heat equation on $B_{g_0}(p,r)$.
\end{lma}

\begin{proof} Let $g_1(x,\tau)=r^{-2}g(x,r^2\tau)$, and $\varphi_1{ (x,\tau)}=r^2\varphi(x,r^2\tau)$ which are defined on $M\times[0, r^{-2}T]$. Then the rescaled Ricci flow and the rescaled function satisfy $\Ric(g_1(\tau))\le \a \tau^{-1}$, $\varphi_1(\tau)\le { \a \tau^{-1}}$, $\varphi_1(0)\le r^2\delta$, $\tR_1(\tau)=\tR(g_1(\tau))=r^2\tR(g(t))$, where $t=r^2\tau$. Moreover, $\varphi_1$ satisfies:
 $$
\lf(\frac{\p }{\p\tau}-\Delta_{g_1(\tau)}\ri)  \varphi_1\leq \tR_1\varphi_1+K\varphi_1^2
$$
in the sense of barrier on $M\times[0,  r^{-2}T]$. Let $G_1(x,\tau;y,u)$ be the heat kernel with respect to $g_1(\tau)$ and let $G(x,t; y,s)$ be the heat kernel with respect to $g(t)$. Then
\bee
r^\frac n2G(x,t;y,s)=G_1(x, \tau; y,u)
\eee
where $t=r^2\tau, s=r^2u$. So
\bee
\begin{split}
\int_{B_{g_1(0)}(p,1)}G_1(x,\tau; y,0)(d\mu_1)_{y,0}=&\int_{B_{g_0}(p,r)}G(x,t; y,0) d\mu_{y,0}
\end{split}
\eee
where $(d\mu)_1$ is the volume element of $g_1$. Therefore it is sufficient to prove the case  of $r=1$.

 Since  $\varphi(t)\le \a t^{-1}\le \a $ for $t\ge 1$, we may assume that $T\le 1$.
Let $$\rho(x)=\sup\{r|\ B_0(x,r)\subset B_0(p,1)\},$$
and set
$$
f(x,t)=\delta\int_{B_0(p,1)}G(x,t;y,0)d\mu_{y,0}.
$$
Then $f(x,0)=\delta$ for $x\in B_0(p,1)$ and $\left(\frac{\partial}{\partial t}-\Delta_{g(t)} \right) f=\tR f$.
Let $A>\delta$ be a constant. Then $ A\rho^{-2}-\varphi>0$ at $t=0$ and near $\p B_0(p,1)$. If $A\rho^{-2}-\varphi<0$ somewhere on $B_0(p,1)\times[0,T]$, then there is $x_0\in B_0(p,1)$, $t_0\le T$ such that
$$
A\rho^{-2}_0=\varphi(x_0,t_0)
$$
and $A\rho^{-2}(x)\ge \varphi(x,t)$ for all $x\in B_0(p,1)\times[0,t_0]$.  Here  $\rho_0=:\rho(x_0)$. Therefore  for $x\in B_0(x_0, \frac12\rho_0)$ and $t\in[0,t_0]$,
$$
\varphi(x,t)\le A\rho^{-2}(x)\le 4A\rho_0^{-2}.
$$
    By the assumption on $\varphi$, we have
 $$
\left(\frac{\partial}{\partial t}-\Delta_{g(t)} \right)\varphi\le \tR \varphi+4AK\rho_0^{-2}\varphi.
$$
in the sense of barrier on $B_0(x_0,\frac12\rho_0)\times[0,t_0]$. Let $b=4AK\rho_0^{-2}$ and
$$
u=e^{-bt}\varphi-f.
$$
Then $u$ satisfies:
 $$
\left(\frac{\partial}{\partial t}-\Delta_{g(t)} \right) u\le \tR u.
$$
in the sense of barrier on $B_0(x_0,\frac12\rho_0)\times[0,t_0]$ and $u(0)\le 0$. By Corollary \ref{c-MP1-1},  for any integer  $l>\max\{\a,K\}+1$, there is $ T_1(l,n,\a,K)\in (0,1]$ such that
$$
 u(x_0,t)\le (\frac12\rho_0)^{-2(l+1)}t^{l}
$$
 for all $t\in [0, t_0\wedge(\frac14\rho_0^2T_1)]$. On the other hand,
 \be\label{e-A}
 A\rho_0^{-2}=\varphi(x_0,t_0)\le \a t_0^{-1}.
 \ee
 Suppose $A\ge 4\a/T_1$, then
 $$t_0\le \frac{\a}A\rho_0^2\le \frac14\rho_0^2T_1\le \frac14\rho_0^2.
 $$ Hence we have

 \bee
 \begin{split}
 e^{-2bt_0}A\rho_0^{-2} -f(x_0,t_0)=&u(x_0,t_0)\\
 \le &   4\rho_0^{-2} \cdot \lf((4\rho_0^{-2})t_0\ri)^l\\
 \le& 4\rho_0^{-2}.
 \end{split}
 \eee

This implies that
$$
A\le \rho_0^2 e^{2bt_0}\lf(  f(x_0,t_0)+4\rho_0^{-2}\ri)\le C_1(\delta\mathcal{S}+1)
$$
 for some $C_1>1$ depending only on $n,\a, K$, because $\rho_0\le 1$ and by \eqref{e-A}
$$
bt_0=4AK\rho_0^{-2}t_0\le 4K\a.
$$

By choosing a larger $C_1=C_1(n,\a, T_1)$ and $A=C_1(\delta(\mathcal{S}+\e)+1)$ with $\e>0$,  we can conclude that
$$
\varphi(x,t)\le A\rho^{-2}(x)
$$
for all $(x,t)\in B_0(p,1)\times[0,T_1\wedge T]$. By letting $\e\to0$  and taking $x=p$, we have
$$
\varphi(p,t)\le C_1(1+\delta \mathcal{S})\rho^{-2}(p)=C_1(1+\delta \mathcal{S}).
$$
The result follows.
\end{proof}

Now we are ready to prove Theorem \ref{t-MP2}.
\begin{proof}[Proof of Theorem \ref{t-MP2}]
As mentioned before, we may assume $T=1$ by parabolic rescaling. Let $R_0, \mu$ be as in Corollary \ref{c-rescaled-heat-esti} with $v_0$ replaced by $e^{v_0}$. And let $T_1$ be the constant obtained from Lemma \ref{l-mp-1}. First consider the case $ R\geq R_0$.  Then the heat kernel $G(x,t;y,s)$ on $B_0(p,\mu^{-1}R)\times B_0(p,\mu^{-1}R)\times[0,T]$ satisfies:
\be\label{e-heat-proof}
G(x,t;y,0)\le C_1t^{-\frac n2}\exp\lf(-\frac{d_0^2(x,y)}{C_1t}\ri)
\ee
for $x, y\in B_0(p,\mu^{-1}R)$ and $t\in [0,T]$. By Lemma \ref{l-mp-1}, we have
$$
\varphi(p,t)\le C_2(\delta \mathcal{S}+R^{-2})
$$
for some constant $C_2=C_2(\a,n, K)$ and $t\in [0, T\wedge T_2 R^2]$, where $T_2=\mu^{-2}T_1 $ depends only on $n, \a, K$ and
$$ \mathcal{S}=\sup_{(x,t)\in B_0(p,\mu^{-1}R)\times[0,T\wedge T_2 R^2]}
 \int_{B_0(p,\mu^{-1}R)}G(x,t;y,0)d\mu_{y,0}.
$$

By \eqref{e-heat-proof}, for $(x,t)\in B_0(p,\mu^{-1}R)\times[0,T\wedge T_2 R^2]$,
\bee
\begin{split}
&\int_{B_0(p,\mu^{-1}R)}G(x,t;y,0)d\mu_{y,0}\\
\le &C_1t^{-\frac n2}\int_{B_0(p,\mu^{-1}R)}\exp\lf(-\frac{d_0^2(x,y)}{C_1t}\ri)d\mu_{y,0}\\
\le&C_1t^{-\frac n2}\int_{B_0(x,2\mu^{-1}R)}\exp\lf(-\frac{d_0^2(x,y)}{C_1t}\ri)d\mu_{y,0}\\
\le&C_1\int_0^{2\mu^{-1}R}\exp\lf(-\frac{r^2}{C_1t}\ri)A(r) dr\\
\le& C_1\lf(t^{-\frac n2}V_0(x,2\mu^{-1}R)\exp\lf(-\frac{4\mu^{-2}R^2}{C_1t}\ri)+2C_1^{-1}t^{\frac n2+1}\int_0^{2\mu^{-1}R}r\exp\lf(-\frac{r^2}{C_1t}\ri)V(r) dr\ri)\\
\le&C_3
\end{split}
\eee
for some $C_3=C_3(n,\a,K,v_0)$. Here we have used the fact that $V(r)=V_0(x,r)\le r^n\exp(v_0 r)$ for $ r>0$. Here $A(r)$ is the area of $\p B_0(x,r)$ with respect to $g_0$.  To summarize, we have
$$
\varphi(p,t)\le C_4(R^{-2}+\delta)
$$
for $t\in [0,T\wedge T_2R^2]$ for some $C_4(n,\a,K,v_0)>0$. If $T>T_2R^2$ and $R^2T_2\le t\le T$, then
$$
\varphi(p,t)\le \a t^{-1}\le \a T_2^{-1}R^{-2}.
$$
This completes the proof of the theorem in the case of $R\geq R_0$.

When $R<R_0$, let $T_3=T_2\wedge R_0^{-2}$. By Corollary \ref{c-rescaled-heat-esti}, the heat kernel $G(x,t;y,s)$ on $B_0(p,\mu^{-1}R)\times B_0(p,\mu^{-1}R)\times [0,T_3R^2\wedge T]$ satisfies the same bound as in \eqref{e-heat-proof}. Now the same argument above shows that for all $t\in [0,T_3 R^2\wedge T]$,
$$\varphi(p,t)\leq C_5(R^{-2}+\delta).$$
For $t\in[T_3R^2,T]$, $\varphi(p,t)\leq C_6R^{-2}$ as $\varphi\leq \a t^{-1}$ for some $C_6(n,\a,K,v_0,\sigma)$. We complete the proof by combining two cases.
\end{proof}

\end{document}